\documentclass[a4paper, 11pt]{amsart}

\usepackage{amsfonts,amsmath,amssymb,amsxtra,amsthm,tikz}
\usepackage{mathrsfs}
\usepackage{times}
\usepackage{anysize}
\usepackage{graphicx}
\marginsize{1in}{1in}{1in}{1in}

\input xy
\xyoption{all}
\usepackage[all,knot]{xy}

\newcommand{\Sub}{\operatorname{Sub}}
\newcommand{\Lk}{\operatorname{Lk}}
\newcommand{\St}{\operatorname{St}}

\newcommand{\ma}{\operatorname{\bf{a}}}

\newcommand{\mc}{\operatorname{\bf{c}}}

\newcommand{\me}{\operatorname{\bf{e}}}
\newcommand{\mf}{\operatorname{\bf{f}}}
\newcommand{\mg}{\operatorname{\bf{g}}}

\newcommand{\mo}{\operatorname{\bf{o}}}
\newcommand{\mpp}{\operatorname{\bf{p}}}
\newcommand{\mq}{\operatorname{\bf{q}}}

\newcommand{\ms}{\operatorname{\bf{s}}}

\newcommand{\mw}{\operatorname{\bf{w}}}

\newcommand{\MQ}{\operatorname{\bf{Q}}}
\newcommand{\MRR}{\operatorname{\bf{R}}}

\newcommand{\MT}{\operatorname{\bf{T}}}
\newcommand{\MUU}{\operatorname{\bf{U}}}

\newtheorem{theorem}{Theorem}[section]
\newtheorem{corollary}[theorem]{Corollary}
\newtheorem{lemma}[theorem]{Lemma}

\newtheorem{problem}[theorem]{Problem}

\newtheorem{conjecture}[theorem]{Conjecture}
\newtheorem{fact}[theorem]{Fact}

\theoremstyle{definition}
\newtheorem{definition}[theorem]{Definition}
\newtheorem{example}[theorem]{Example}
\newtheorem{remark}[theorem]{Remark}

\title{Subword complexes and edge subdivisions} 
\author{Mikhail Gorsky} 

\address{Steklov Mathematical Institute,  
8 Gubkina Street, Moscow, Russia 119991.}
\address{Universit\'e Paris Diderot -- Paris 7, UFR de
Math\'ematiques, Case 7012, Institut de Math\'ematiques de Jussieu -- Paris Rive Gauche, UMR 7586 du CNRS,
B\^at. Sophie Germain, 75205 Paris Cedex 13, France}
\email{mgorsky@math.jussieu.fr}

\begin{document}

\begin{abstract}
For a finite Coxeter group, a subword complex is a simplicial complex associated with a pair $(\MQ, \pi),$ where $\MQ$ is a word in the alphabet of simple reflections, $\pi$ is a group element. We discuss the transformations of such a complex induced by braid moves of the word $\MQ.$ We show that under certain conditions, this transformation is a composition of edge subdivisions and inverse edge subdivisions. In such a case, we describe how the $H-$ and the $\gamma-$polynomials change under this operation. This case includes all braid moves for groups with simply-laced Coxeter diagrams.
\end{abstract}

\maketitle


\section{Introduction}
Subword complexes were introduced by A.~Knutson and E.~Miller in \cite{KM1} in the context of Schubert polynomials and matrix Schubert varieties. It was soon realized that they are interesting from the point of view of Coxeter combinatorics. A subword complex is associated with a pair $(\MQ, \pi),$ where $\MQ$ is a word in the alphabet of simple reflections and $\pi$ is an element of a Coxeter group $W.$ The simplices in the complex correspond to the subwords in $\MQ$ whose complements contain reduced expressions of $\pi.$ The exchange axiom arises as the transition between two adjacent maximal simplices. In \cite{KM1}, subword complexes were shown to be vertex-decomposable and, therefore, shellable. This provides a new proof (and a new interpretation) of the Cohen-Macaulayness for matrix Schubert varieties and also for ordinary Schubert varieties, cf. \cite{KM1}. Using the shellability, in \cite{KM2}, Knutson and Miller proved that an arbitrary subword complex is homeomorphic to a sphere or to a ball. They also give the sufficient and necessary condition on $(\MQ, \pi)$ for which the complex is spherical. 

For a spherical subword complex it is natural to ask, whether it is polar dual to some simple polytope. The answer to the general question is still unknown, but in some cases such polytopes are constructed. One of the constructions is due to V.~Pilaud and C.~Stump \cite{PS} whose polytopes are called {\it brick}. As polytopes dual to the subword complexes there arise such classical polytopes as generalized associahedra and certain cyclic polytopes, cf. \cite{CLS}\cite{PS}. For a certain class of words, subword complexes turn out to be isomorphic to cluster complexes introduced by S.~Fomin and A.~Zelevinsky \cite{FZ}. Moreover, some important objects from theory of cluster algebras, such as $c-$vectors, can be interpreted via natural geometric realizations of these subword complexes and their fans \cite{CLS}. One class of subword complexes is called multi-cluster complexes. A number of properties of the category of representations of the corresponding quiver have a natural interpretation in terms of combinatorics of these complexes, cf. \cite{CLS}. 

In this article, we discuss the following question: given two words $\MQ_1$ and $\MQ_2$ expressing the same group element, what can we say about the relation between the corresponding subword complexes $\Delta_1$ and $\Delta_2?$ Theorem \ref{widetildeiso} and Corollary \ref{delta2iso} provide a partial answer to this question. For $\MQ$ and $\MQ'$ related by exactly one braid move, we show which simplices should be removed from $\Delta_1$ and which should be added, in order to obtain $\Delta_2.$ Unfortunately, in general, a lot of simplices might be counted several times. Therefore, this result, while being constructible, cannot be applied directly to calculations of enumerative polynomials and does not reflect the picture quite precisely. Nonetheless, under certain conditions, we prove that this operation is nothing but a composition of edge subdivisions and inverse edge subdivisions, which can be simply described. This is done in Theorem \ref{braidstellar1} and \ref{braidstellarmij3}. In particular, this holds for any pair of words related by a braid move in the simply-laced case. In Theorem \ref{braid2trunc}, we give the dual result in terms of polytopes; the edge subdivisions are replaced by truncations of faces of codimension $2,$ or simply $2-$truncations.

Edge subdivisions and $2-$truncations attracted the attention recently because of their applications to enumerative polynomials. To any simplicial complex homeomorphic to a sphere, one associates the so-called $\gamma-$polynomial whose coefficients are certain linear combinations of the numbers $f_k$ of simplices of dimension $k.$ In \cite{Ga}, Gal conjectured that all the coefficients of the $\gamma-$polynomial of an arbitrary flag spherical simplicial complex are non-negative. N.~Aisbett proved this conjecture for any complex which can be obtained from the nerve complex of a cube by a sequence of edge subdivisions. She also proved a generalization concerning the realization of these $\gamma-$polynomials as $f-$polynomials of some other complexes, cf. \cite{A} and references therein. The same result, but in the dual terms of $2-$truncations, has been shown by Volodin, cf. \cite{BV}\cite{V1}\cite{V2}. The class of polytopes which can be obtained from the cube of a fixed dimension by a sequence of $2-$truncations turns out to be very interesting. In particular, it contains all flag nestohedra, graph-associahedra and generalized associahedra of types $ABCD;$ details can be found in \cite{BV}. Our main results give a new interpretation of $2-$truncations and edge subdivisions in very natural combinatorial terms; indeed, we see that, in some cases, these operations correspond just to braid moves in the Coxeter systems. In the upcoming sequel \cite{Go2}, we will apply this to the subword complexes $\Delta(\mc\mw_o; w_o),$ where $\mc$ is a reduced expression of a Coxeter element, $\mw_o$ is an arbitrary reduced expression of the longest element $w_o.$ Our results also enable us to define various partial orders on sets of reduced expressions of any word in a Coxeter system, cf. the end of Section 5. 
 
The paper is organized as follows. In Section 2 we recall the notion of edge subdivisions of simplicial complexes and of $2-$truncations of simple polytopes. We give also the necessary material on enumerative polynomials. Section 3 and Section 4 contain the introductions to finite Coxeter systems and to subword complexes, respectively. In Section 5, we formulate and prove all new results.

This is a part of my Ph.D. project at Steklov Mathematical Institute. I am very grateful to my doctoral adviser Prof. Victor M. Buchstaber for the inspiration in the work and for his support and patience. Many thanks to Jean-Philippe Labb\'e, Vincent Pilaud and Salvatore Stella for explaining me the nature of the subword complexes. The work was supported by DIM RDM-IdF of the R\'{e}gion \^{I}le-de-France.

\section{Edge subdivisions and 2-truncations}
Given a simplex $\sigma$ in a simplicial complex $X,$ the {\it link} and the {\it star} of $\sigma$ are the following subcomplexes:
$$\Lk_X(\sigma) = \left\{\rho \in X | \sigma \cup \rho \in X, \sigma \cap \rho = \emptyset\right\};$$
$$\St_X(\sigma) = \left\{\rho \in X | \sigma \cup \rho \in X\right\}.$$

Let us define an {\it edge subdivision} of a simplicial complex, following \cite{Ga}. This operation is also called a $1-${\it stellar subdivision}.

\begin{definition} 
Let $X$ be a simplicial complex. Let $\eta = \left\{s, t\right\}$ be an edge. Define $\Sub_{\eta}(X)$ to be
a simplicial complex constructed from $X$ by bisection of all simplices containing $\eta.$ In other words,
let $r$ be any letter not in the vertex set of $X.$ Then $S \cup \left\{r\right\}$ is the vertex set of $\Sub_{\eta}(X)$ and
$$\Sub_{\eta}(X) = \left\{\sigma|\eta \not\subset \sigma \in X\right\} \cup \left\{\sigma \cup \left\{r\right\}, \sigma \cup \left\{r, s\right\}, \sigma \cup \left\{r, t\right\} | \sigma \in \Lk_X(\eta)\right\} = $$
$$= \left\{\sigma|\eta \not\subset \sigma \in X\right\} \cup \left\{\sigma \cup \left\{r\right\} | \sigma \in \partial \St_X(\eta) \right\}.$$
We say that $\Sub_{\eta}(X)$ is a subdivision of $X$ along $\eta.$
\end{definition}

We will use the following natural generalization of this transformation.

\begin{definition}
Let $X$ be a simplicial complex. Let $\eta = \left\{s, t\right\}$ be an edge. Define $\Sub^k_{\eta}(X)$ to be
a simplicial complex constructed from $X$ by dividing all simplices containing $\eta$ in $n$ parts. In other words, let $r_1,\ldots,r_k$ be any $k$ letters not in the vertex set $S$ of $X.$ Then $S \cup \left\{r_1,\ldots,r_k\right\}$ is the vertex set of $\Sub^k_{\eta}(X)$ and
$$\Sub^k_{\eta}(X) = \left\{\sigma|\eta \subset \sigma \in X\right\} \cup \left\{\sigma \cup \left\{r_i\right\}, \sigma \cup \left\{r_{i-1}, r_i\right\}, \sigma \cup \left\{r_k, t\right\} | i \in \left\{1,\ldots,k\right\}, \sigma \in \Lk_X(\eta)\right\},$$
where we pose $r_0 = s.$
We say that $\Sub^k_{\eta}(X)$ is a $k-$subdivision of $X$ along $\eta.$
\end{definition}

The operation of the edge $k-$subdivision is just a composition of $k$ subsequent edge subdivisions of $X,$ where the $i-$th subdivision is made along the edge $\left\{r_{i-1}, t\right\},$ with the new vertex $r_i.$

A convex polytope of dimension $n$ is said to be {\it simple} if each of its vertices is contained in precisely $n$ facets. A polytope $P$ is simple if and only if its polar dual polytope $P^*$ is {\it simplicial}, i.e. each face of $P^*$ is a simplex.
The boundary of a simplicial $n-$dimensional polytope is a simplicial complex
of dimension $(n - 1).$ For a simple polytope $P,$ we shall denote by $K_P$ the
boundary complex $\partial P^*$ of the dual polytope. It coincides with the nerve of the
covering of $\partial P$ by the facets. That is, the vertices of $K_P$ are the facets of $P,$ and a
set of vertices spans a simplex whenever the intersection of the corresponding facets
is nonempty. We say that $K_P$ is the {\it nerve complex} of $P,$ or that $K_P$ is {\it polar dual} to $P.$
For simple polytopes, there is a dual operation to the edge subdivision: 
this is the truncation of a face of codimension $2,$ or simply the $2-${\it truncation}. 

\begin{definition}
Let $G$ be a face of codimension $2$ of a simple (combinatorial) polytope $P.$ Let $K_P = \partial P^*$ be its nerve complex and $J \in P*$ be the face dual to $G.$ We say that the polytope $\widetilde{P},$ such that $K_{\widetilde{P}} = \Sub_J(K_P),$ is the truncation of $P$ at $G.$ Such a polytope $\widetilde{P}$ exists and is unique, up to combinatorial isomorphism.
\end{definition}

Assume that $P$ is $n-$dimensional. Geometrically, the polytope $\widetilde{P}$ can be obtained from a realization of $P$ by intersecting the latter with a new half-space $H,$ such that the intersection of the $(n-1)$-dimensional plane $h = \partial H$ with $\partial P$ is precisely the link of $G$ in $\partial P.$ It means that the facets $\widetilde{P}$ are precisely the facets of $P$ and there is a one new facet $K$ isomorphic to $I \times G,$ and such that $\Lk_{\widetilde{P}}(K) = \Lk_P(G).$ Here $I$ is the closed interval $[0,1].$ More detailed treatment of $2-$truncations can be found in \cite{BP,BV}.

We see that, for the nerve complex of a simple polytope $P,$ an edge $k-$subdivision is dual to a composition of $k$ subsequent $2-$truncations of $P,$ where:
\begin{itemize}
\item{} the first truncation is made at a certain face $G,$ the new facet $G_1$ is isomorphic to $I \times G;$
\item{} the $(i+1)-$th truncation is made at a facet $\left\{1\right\} \times G$ of $G_{i},$ the new facet $G_{i+1}$ is again isomorphic to $I \times G.$ 
\end{itemize}

A simple polytope $P$ is {\it flag} if any set of pairwise intersecting facets $F_{i_1},\ldots, F_{i_k}$ has a non-empty intersection. Dually, a simplicial complex $X$ is {\it flag}, if any clique $T$ in the vertex set of $X$ (i.e., any two distinct vertices of $T$ are joined by an edge) is a face of $X.$ It is evident that a simple polytope $P$ is flag if and only if so is its nerve complex $K_P.$ We have the following fact relating edge subdivisions and $2-$truncations with the flagness property.

\begin{fact} [{\cite[Proposition~2.4.6]{Ga}}]
\begin{itemize}
\item[(i)]  If a simplicial complex is flag, then so is any its edge subdivision;
\item[(ii)] If a simple polytope is flag, then so is any its $2-$truncation.
\end{itemize}
\end{fact}

The $f-${\it{vector}} of an $(n - 1)-$dimensional simplicial complex
$X$ is $f(K) = (f_0, f_1,\ldots, f_{n-1}),$ where $f_i$ is the number of $i-$dimensional simplices
in $X.$ We also set here $f_{-1} = 1.$ The $h-${\it{vector}} $h(X) = (h_0, h_1,\ldots, h_n)$ is defined by the
identity
$$h_0 s^n + h_1 s^{n-1} + \ldots + h_n = (s - 1)^n + f_0 (s - 1)^{n-1} + \ldots + f_{n-1}.$$
We define the $H-${\it{polynomial}} of $X$ as the generating function in two variables:
$$H(X)(\alpha, t) = \sum\limits_{k=1}^n h_k \alpha^k t^{n-k}.$$
A {\it (simplicial) generalized homology sphere} of dimension $n$ is a simplicial
complex such that the link of any simplex $\sigma$ has the homology of a sphere of dimension
$(n - \dim\sigma).$ We will omit the superscript if not necessary.
It is known (under the name of the Dehn-Sommerville relations) that if $X$ is a generalized homological sphere, then $H(X)(\alpha, t)$ is symmetric. Therefore, it can be rewritten as a polynomial in $\alpha t$ and $(\alpha + t);$ let us denote the coefficients of this polynomial by $\gamma_0, \gamma_1,\ldots, \gamma_{[\frac{n}{2}]}:$
$$H(X)(\alpha, t) = \sum\limits_{k=1}^{[\frac{n}{2}]} \gamma_k (\alpha t)^k (\alpha + t)^{([\frac{n}{2}] - k)}.$$
We define the $\gamma-$polynomial as the generating function in one variable $\tau:$
$$\gamma(X)(\tau) = \sum\limits_{k=1}^{[\frac{n}{2}]} \gamma_k \tau^k.$$
We pose $H(\emptyset) = \gamma(\emptyset) = 0.$ 
Gal described how the $H-$ and the $\gamma-$polynomial change under edge subdivisions.

\begin{fact} {\cite[Proposition~2.4.3]{Ga}} \label{polysubdiv} 
\begin{itemize}
\item[(i)] $H(\Sub_{\eta}(X))(\alpha, t) = H(X)(\alpha, t) + \alpha tH(\Lk_X(\eta))(\alpha, t);$
\item[(ii)] If $X$ is a generalized homological sphere, then we have
$$\gamma(\Sub_{\eta}(X))(\tau) = H(X)(\tau) + \tau H(\Lk_X(\eta))(\tau).$$
\end{itemize}
\end{fact}
This motivated Gal to formulate the following conjecture. 

\begin{conjecture} \label{galconj}
If $X$ is a generalized homology sphere, then $\gamma(X)$ has nonnegative coefficients.
\end{conjecture}
This is a generalization of the important Charney-Davis conjecture \cite{CD} equivalent to the non-negativity of the highest coefficient of $\gamma(X).$ We see that if Conjecture \ref{galconj} holds for $X$ and $\Lk_X(\eta),$ then it also holds for $\Sub_{\eta}(X).$ 

\section{Coxeter groups}

We consider a {\it finite Coxeter group} $W$ acting on a $n-$dimensional euclidean space $V,$ that is, a finite group generated
by reflections. The set of reflections in $W$ is denoted by $R.$ The {\it Coxeter
arrangement} of $W$ is the collection of all reflecting hyperplanes. Its complement
in $V$ is a union of open polyhedral cones. Their closures are called {\it chambers}.
The Coxeter fan is the polyhedral fan formed by the chambers together with all
their faces. This fan is {\it complete} (its cones cover $V$) and {\it simplicial} (all cones are
simplicial), and we can assume without loss of generality that it is {\it essential} (the
intersection of all chambers is reduced to the origin). We fix an arbitrary chamber
$C$ which we call the {\it fundamental chamber}. The {\it simple reflections} of $W$ are the $n$
reflections orthogonal to the facet defining hyperplanes of $C.$ The set $S = \left\{s_1,\ldots,s_n\right\} \subset R$ of
simple reflections generates $W,$ so the choice of $S$ is equivalent to the choice of $C.$ The
pair $(W; S)$ forms a {\it Coxeter system}. For simple reflections $s_i, s_j \in S,$ denote by $m_{ij}$ the order
of the product $(s_i s_j)$ in W. A Coxeter group is simply-laced, if for any $i, j,$ we have $m_{ij} \in \left\{2, 3\right\}.$

The {\it length} $l(w)$ of an element $w \in W$ is the length of the smallest
expression of $w$ as a product of the generators in $S.$ An expression $w = w_1 w_2 \ldots w_p$
with $w_1,\ldots,w_p \in S$ is called {\it reduced} if $p = l(w).$ We denote by $w_o$ the unique longest element in $W,$ it is known to be unique.

We denote by $S^*$ the set of words on the alphabet $S,$ and by $\me$ the empty word. We can consider $S^*$ as the free monoid with the generating set $S$ and the concatenation as the operation.
To avoid confusion, we denote with a square letter $\mw$ the letter of the alphabet $S$
corresponding to the single reflection $s \in S.$ Similarly, we use a square letter like $\mw$
to denote a word of $S^*,$ and a normal letter like $w$ to denote its corresponding group
element in $W.$ For example, we write $\mw := \mw_1\ldots \mw_p$ meaning that the word $\mw \in  S^*$
is formed by the letters $\mw_1,\ldots,\mw_p,$ while we write $w := w_1 \ldots w_p$ meaning that
the element $w \in W$ is the product of the simple reflections $w_1,\ldots,w_p.$

On $S^*,$ there are two types of operations reflecting the group structure of $W.$ A {\it nil-move} removes two consecutive identical letters $\ms_i\ms_i$ from a word $\mw \in S^*,$ for some $i.$ If $\mw$ contains a subword $\ms_i\ms_j\ms_i\ms_j\ms_i\ldots$ of length $m_{ij},$ then there is a {\it braid-move} transforming $\mw$ into $\mw'$ by changing $\ms_i\ms_j\ms_i\ms_j\ms_i\ldots$ by the subword $\ms_j\ms_i\ms_j\ms_i\ms_j\ldots$ of the same length $m_{ij}.$ Note that neither a nil-move, nor a braid-move changes a group element $w \in W$ expressed by $\mw.$ Recall the {\bf Word Property} which holds for any Coxeter system $(W; S).$

\begin{fact} [{\cite[Theorem~3.3.1]{BB}}] \label{wordproperty}
\begin{itemize}
\item[(1)] Any expression $\mw$ for $w \in W$ can be transformed into a
reduced expression for $w$ by a sequence of nil-moves and braid-
moves.
\item[(2)] Every two reduced expressions for $w$ can be connected via a
sequence of braid-moves.
\end{itemize}
\end{fact}

A {\it Coxeter element} $c$ is a product of
all simple reflections in some order. We choose an arbitrary
reduced expression $\mc$ of $c$ and denote by $\mw(\mc)$ the {\it $\mc$-sorting word} of $w,$ that is the
lexicographically first (as a sequence of positions) reduced subword of ${\mc}^{\infty} = \mc\mc\mc\ldots$ for $w.$ In
particular, ${\mw}_{\mo}(\mc)$ denotes the $\mc$-sorting word of the longest element $w_o \in W.$

\section{Subword complexes}

Let $(W; S)$ be a Coxeter system, let $\MQ := \mq_1\ldots\mq_m$ belong to $S^*$
and let $\pi$ be an element in $W.$ The {\it subword
complex} $\Delta(\MQ; \pi)$ is the pure simplicial complex of subwords of $\MQ,$ whose complements
contain a reduced expression of $\pi.$ The vertices of this simplicial complex are labeled
by (positions of) the letters in the word $\MQ.$ Note that two positions are different
even if the letters of $\MQ$ at these positions coincide. 
The maximal simplices of the subword complex $\Delta(\MQ; \pi)$ are the complements of reduced expressions of $\pi$ in the word $\MQ.$

In \cite{KM1}, it was shown that the subword complex $\Delta(\MQ; \rho)$ is either a triangulated sphere (or simply {\it spherical}), or a triangulated ball. It is spherical if and only if the {\it Demazure product} of $\MQ$ is equal to $\pi,$ see the definition of this product and the proof in \cite[Section~3]{KM1}. In this case, $H(\Delta(\MQ; \rho))$ is symmetric and $\gamma(\Delta(\MQ; \rho))$ is defined. In some generality, these spherical subword complexes are polar dual to some simple polytopes. The general description of these polynomials and their geometric realizations is not yet known; under certain conditions, such realizations are constructed and called the {\it brick polytopes} in \cite{PS}. We are interested in combinatorial types of these polytopes. The facets of such a polytope are labeled by (positions of) the letters in the word $\MQ,$ whose complements contain a reduced expression of $\pi.$ Let us give some examples. Denote for the Coxeter group $A_n$ the simple reflections by $s_1,\ldots,s_n$ in a natural order. 

\begin{example} \label{cluster} (Cluster complexes and generalized associahedra).
For any $\mbox{c},$ the subword complex $\Delta(\mc{\mw}_o(\mw{c}); w_o)$ coincides with the cluster complex of type $W.$ Its dual polytope is a generalized associahedron of type $W.$ For example, for 
$\mc = \ms_1\ms_2\ldots\ms_n,$ we have
$$\mc{\mw}_o(\mc) = \ms_1\ms_2\ldots\ms_n\ms_1\ms_2\ldots\ms_n\ms_1\ms_2\ldots{\ms}_{n-1}\ms_1\ms_2\ldots\ms_{n-2}\ldots\ms_1\ms_2\ms_1.$$
Every position yields a vertex of the complex and the facet of the polytope. 
\end{example}

\begin{example} (Duplicated word).
Let $P$ be a set of $n$ positions in some reduced expression $\mw_o$ of the longest element. Define a new word $\MQ^{dup} \in S^*$ obtained by duplicating the letters of $\mw_o$ at positions in $P.$ Under certain conditions (see \cite[Example~3.8]{PS}), the complex $\Delta(\MQ^{dup}; w_o)$ is polar dual to the $n-$dimensional cube; in other words, it is the boundary complex of the $n-$dimensional cross-polytope. For example, the word 
$$\ms_1\ms_1\ms_2\ms_2\ldots\ms_n\ms_n\ms_1\ms_2\ldots\ms_{n-1}\ms_1\ms_2\ldots\ms_{n-2}\ldots\ms_1\ms_2\ms_1$$
is duplicated for $P = \left\{1,\ldots,n\right\}.$ First $2n$ positions yield vertices of $\Delta(\MQ^{dup}; w_o),$ so they correspond to the facets of the cube; the facets associated with the positions $2k - 1$ and $2k$ of letter $\ms_k$ are opposite to each other.
\end{example}

\begin{example} (Multi-cluster complexes).
For any $\mc,$ the subword complex $\mathcal{S}({\mc}^k\mw_o(\mc))$ is called a {\it $k$-cluster complex} of type $W.$ It is not known in general, whether thus defined complexes are polar dual to some polytopes. When this holds, the polytope polar dual to $\mathcal{S}(\mbox{c}^k\mbox{w}_o(c))$ is called the {\it $k-$associahedron} of type $W.$ See {\cite{CLS}\cite{PS}} for more details.
\end{example}
We will use the following simple observation.

\begin{lemma} \label{linksubword}
The link of a simplex corresponding to a word $\ma_1 \ma_2\ldots\ma_x$ in a subword complex $\Delta(\MUU; \rho),$ where $\MUU \in S^*, \rho \in W$ is isomorphic to the complex $\Delta(\MUU'; \rho),$ where $\MUU'$ is obtained from $\MUU$ by removing all the letters $\ma_i, i = 1,2,\ldots,x.$
\end{lemma}

\begin{proof}
By definition of the link, the claim says that if $\ma_1 \ma_2\ldots\ma_x$ is a subword of $\MT$ and $\MT$ is a subword of $\MUU,$ than the complement $\MT^c$ of $\MT$ in $\MUU$ contains a reduced expression of $\rho$ if and only if the complement ${\MT^c}'$ to $\MT'$ in $\MUU'$ contains a reduced expression of $\rho,$ where $\MT'$ is obtained from $\MT$ by removing all the letters $\ma_i, i = 1,2,\ldots,x.$ But ${\MT^c}'$ coincides with $\MT^c,$ and the statement follows.
\end{proof}

\section{Braid moves and edge subdivisions}

For any $s_i, s_j \in S$ we define a family of words in $S^*:$ 
$$\mw_{i,j}^{k}: = \ms_i\ms_j\ms_i\ms_j\ms_i \ldots,$$
containing $(m_{ij} - k) = (m_{ji} - k)$ letters, for $k \in \left\{0, 1,\ldots, m_{ij}\right\}.$ We have
$$\mw_{i,j}^k = \ms_i\ms_j\mw_{i,j}^{k+1}.$$
If $k$ is even, we also have $\mw_{i,j}^k = \ms_i\mw_{i,j}^{k+1}\ms_j,$ otherwise $\mw_{i,j}^k = \ms_i\mw_{i,j}^{k+1}\ms_i.$
The braid relation for $s_i$ and $s_j$ is equivalent to the identity 
$$w_{i,j}^{0} = w_{j,i}^{0},$$
where $w_{i,j}^k$ and $w_{j,i}^k$ are the elements of $W$ expressed by the words $\mw_{i,j}^k$ and $\mw_{j,i}^k,$ respectively.
Fix any two words $\MQ, \MQ' \in S^*.$ The words $\MQ \mw_{i,j}^{0} \MQ'$ and $\MQ' \mw_{i,j}^{0} \MQ$ express the same element in the group $W$ and are linked to each other by only one braid move. In this section we shall discuss the relation between the corresponding subword complexes. Denote
$$\MQ_1^k := \MQ \mw_{i,j}^{k} \MQ', \quad \MQ_2^k := \MQ \mw_{j,i}^{k} \MQ'.$$
Let $\mf_l$ and $\mg_l$ be the $l-$th letters in $\mw_{i,j}^0$ and $\mw_{j,i}^{0},$ respectively, for $l = 1,2,\ldots,m_{ij}.$ Here we consider $w_{i,j}^0$ as a subword of $\MQ_1^0$ and  $\mw_{j,i}^0$ as a subword of $\MQ_2^0.$ We call the letters $\mf_l$ and $\mg_l$ {\it internal} if $l \in \left\{2,3,\ldots,(m_{ij} - 1)\right\}.$
Consider $1-$simplices $F = \left\{\mf_1, \mf_{m_{ij}}\right\}$ and $G = \left\{\mg_1, \mg_{m_{ij}}\right\}.$
We fix an arbitrary group element element $\pi \in W$ and consider the subword complexes
$$\Delta_1 := \Delta(\MQ_1^0; \pi), \quad \Delta_2 := \Delta(\MQ_2^0; \pi).$$
Consider a family of conditions indexed by $k \in \left\{0, 1,\ldots, m_{ij}\right\}:$

\begin{itemize}
\item[($A_k$)] $\MQ_1^k$ does not contain any reduced expression of $\pi;$
\item[($B_k$)] $\MQ_2^k$ does not contain any reduced expression of $\pi.$
\end{itemize}

Note that ($A_k$) implies ($A_l$) and ($B_l$), for any $l > k.$ Similarly, ($B_k$) implies ($A_l$) and ($B_l$), for any $l > k.$ The following lemma is an immediate consequence of the definition of subword complexes.

\begin{lemma} \label{A2B2}
Condition ($A_2$) is satisfied if and only if the simplex $F$ does not belong to $\Delta_1.$ Similarly, condition ($B_2$) is satisfied if and only if $G \notin \Delta_2.$
\end{lemma}

Fix any $l = 1,2,\ldots,(m_{ij}-1).$ Let
$$\varphi_l: \Delta(\MQ_1^2; \pi) \overset\sim\to \Lk_{\Delta_1}(\left\{\mf_l, \mf_{l+1}\right\});
\quad \psi_l: \Delta(\MQ_2^{2}; \pi) \overset\sim\to \Lk_{\Delta_2}(\left\{\mg_l, \mg_{l+1}\right\});$$
$$\varphi_G: \Delta(\MQ_1^2; \pi) \overset\sim\to \Lk_{\Delta_2}(G);
\quad \psi_F: \Delta(\MQ_2^2; \pi) \overset\sim\to \Lk_{\Delta_1}(F);$$
be the isomorphisms introduced in Lemma \ref{linksubword}.

\begin{lemma} \label{linkflgl}
For every internal $\mf_l, \mg_l,$ we have
\begin{equation} \label{linkfl}
\Lk_{\Delta_1}(\left\{\mf_l\right\}) = (\Lk_{\Delta_1}(\left\{\mf_{l-1}, \mf_l\right\}) * \left\{\mf_{l-1}\right\}) \cup (\Lk_{\Delta_1}(\left\{\mf_l, \mf_{l+1}\right\}) * \left\{\mf_{l+1}\right\}) = 
\end{equation}
$$(\varphi_{l-1}(\Delta(\MQ_1^2; \pi)) * \left\{\mf_{l-1}\right\}) \cup (\varphi_l(\Delta(\MQ_1^2; \pi)) * \left\{\mf_{l+1}\right\}); $$
\begin{equation} \label{linkgl}
\Lk_{\Delta_2}(\left\{\mg_l\right\}) = (\Lk_{\Delta_2}(\left\{\mg_{l-1}, \mg_l\right\}) * \left\{\mg_{l-1}\right\}) \cup (\Lk_{\Delta_2}(\left\{\mg_l, \mg_{l+1}\right\}) * \left\{\mg_{l+1}\right\}) = 
\end{equation}
$$(\psi_{l-1}(\Delta(\MQ_2^2; \pi)) * \left\{\mg_{l-1}\right\}) \cup (\psi_l(\Delta(\MQ_2^2; \pi)) * \left\{\mg_{l+1}\right\}).$$
\end{lemma}

Note that these unions are not necessary disjoint.

\begin{proof}
By Lemma \ref{linksubword}, $\Lk_{\Delta_1}(\left\{\mf_l\right\})$ is isomorphic to the subword complex $\Delta(\MRR_l; \pi),$ where $\MRR_l$ is obtained from $\MQ_1^0$ by removing the letter $\mf_l:$
$$\MRR_l = \MQ\mf_1\mf_2\ldots\mf_{l-1}\mf_{l+1}\ldots\mf_{m_{ij}}\MQ'.$$
The letters on the positions $\mf_{l-1}$ and $\mf_{l+1}$ coincide. No reduced expression of $\pi$ can contain two consecutive identical letters; thus, the complement to every reduced expression of $\pi$ in $\MRR_l$ contains at least one of $\mf_{l-1}$ and $\mf_{l+1}.$ Therefore, any maximal simplex in $\Delta(\MRR_l; \pi)$ is contained in the union $(\Lk_{\Delta(\MRR_l;\pi)}(\left\{\mf_{l-1}\right\}) * \left\{\mf_{l-1}\right\}) \cup (\Lk_{\Delta(\MRR_l;\pi)}(\left\{\mf_{l+1}\right\}) * \left\{\mf_{l+1}\right\}).$
Since each simplex in $\Delta(\MRR_l; \pi)$ is contained in a maximal one, we have 
$$\Delta(\MRR_l; \pi) \subset (\Lk_{\Delta(\MRR_l;\pi)}(\left\{\mf_{l-1}\right\}) * \left\{\mf_{l-1}\right\}) \cup (\Lk_{\Delta(\MRR_l;\pi)}(\left\{\mf_{l+1}\right\}) * \left\{\mf_{l+1}\right\}).$$
The inverse inclusion is clear; thus, we have
$$\Delta(\MRR_l; \pi) = (\Lk_{\Delta(\MRR_l;\pi)}(\left\{\mf_{l-1}\right\}) * \left\{\mf_{l-1}\right\}) \cup (\Lk_{\Delta(\MRR_l;\pi)}(\left\{\mf_{l+1}\right\}) * \left\{\mf_{l+1}\right\}),$$
that implies the first identity in (\ref{linkfl}). The second identity follows from the definition of the isomorphisms $\varphi_{l-1}, \varphi_l.$ The proof of identities (\ref{linkgl}) is similar.
\end{proof}

Define the subcomplexes $\Delta_{1,int}, \Delta_{1, F} \in \Delta_1, \Delta_{2,int}, \Delta_{2, F} \in \Delta_2$ as follows:
$$\Delta_{1,int} = \left\{\varphi_l(\sigma) \cup \left\{\mf_l, \mf_{l+1}\right\}, \varphi_l(\sigma) \cup \left\{\mf_l\right\}, \varphi_{l-1}(\sigma) \cup \left\{\mf_l\right\} | \sigma \in \Delta(\MQ_1^2; \pi), l = 2, 3,\ldots,(m_{ij} - 1) \right\};$$
$$\Delta_{1,F} =  \left\{\psi_F(\rho) \cup \left\{\mf_1, \mf_{m_{ij}}\right\} | \rho \in \Delta(\MQ_2^2; \pi) \right\};$$
$$\Delta_{2,int} = \left\{\psi_l(\rho) \cup \left\{\mg_l, \mg_{l+1}\right\}, \psi_l(\rho) \cup \left\{\mg_l\right\}, \psi_{l-1}(\rho) \cup \left\{\mg_l\right\} | \rho \in \Delta(\MQ_2^2; \pi), l = 2, 3,\ldots,(m_{ij} - 1) \right\};$$
$$\Delta_{2,G} =  \left\{\varphi_F(\rho) \cup \left\{\mg_1, \mg_{m_{ij}}\right\} | \sigma \in \Delta(\MQ_1^2; \pi) \right\}.$$

Consider the complexes
$$\widetilde{\Delta_1} = \Delta_1 \backslash (\Delta_{1,int} \cup \Delta_{1,F}); \quad \widetilde{\Delta_2} = \Delta_2 \backslash (\Delta_{2,int} \cup \Delta_{2,G}).$$

\begin{lemma} \label{A3B3edges}
Assume that conditions ($A_3$) and ($B_3$) hold and $m_{ij} > 3.$ Then $\Delta_1$ contains no $1-$simplices of the form $\left\{\mf_k, \mf_l\right\},$ for $\mf_k$ internal, $|k - l| \neq 1.$
\end{lemma}

\begin{proof}
Using the same argument about consecutive identical letters as in the proof of Lemma \ref{linkflgl}, one can show the following:
\begin{itemize}
\item[(i)] for $1 < k < l < m_{ij}, \quad \left\{\mf_k, \mf_l\right\}$ belongs to $\Delta_1$ if and only if so does $\left\{\mf_{k-1}, \mf_k, \mf_l, \mf_{l+1}\right\};$
\item[(ii)] for $1 < k < m_{ij}, \quad \left\{\mf_k, \mf_{m_{ij}}\right\}$ belongs to $\Delta_1$ if and only if so does $\left\{\mf_{k-1}, \mf_k, \mf_{m_{ij}}\right\};$
\item[(iii)] for $1 < k < m_{ij}, \quad \left\{\mf_1, \mf_k\right\}$ belongs to $\Delta_1$ if and only if so does $\left\{\mf_1, \mf_{k}, \mf_{k+1}\right\}.$
\end{itemize}
By arguments from the proof of the Lemma \ref{linksubword}, this implies that
\begin{itemize}
\item[(i)] for $1 < k < l < m_{ij}, \quad \left\{\mf_k, \mf_l\right\}$ belongs to $\Delta_1$ if and only if ($A_4$) does not hold;
\item[(ii)] for $1 < k < m_{ij}, \quad \left\{\mf_k, \mf_{m_{ij}}\right\}$ belongs to $\Delta_1$ if and only if ($A_3$) does not hold;
\item[(iii)] for $1 < k < m_{ij}, \quad \left\{\mf_1, \mf_k\right\}$ belongs to $\Delta_1$ if and only if ($B_3$) does not hold.
\end{itemize}
It remains to recall that ($A_3$) implies ($A_4$).
\end{proof}

\begin{corollary} \label{A3B3disjoint}
Assume that conditions ($A_3$) and ($B_3$) hold and $m_{ij} > 3.$ Then the unions in identities (\ref{linkfl})-(\ref{linkgl}) are disjoint.
\end{corollary}

\begin{proof}
By Lemma \ref{A3B3edges}, under our assumptions, the $1-$simplex $\left\{\mf_{l-1}, \mf_{l+1}\right\}$ does not belong to $\Delta_1,$ and $\left\{\mg_{l-1}, \mg_{l+1}\right\}$ does not belong to $\Delta_2.$ This is a reformulation of the claim.
\end{proof}

\begin{lemma} \label{widetildesense}
The complex $\widetilde{\Delta_1}$ (resp. $\widetilde{\Delta_2}$)  is the maximal subcomplex of $\Delta_1$ (resp. $\Delta_2$), containing neither $F$ (resp. $G$), nor any vertex given by an internal letter $\mf_l$ (resp $\left\{\mg_l\right\}$).
\end{lemma} 

\begin{proof}
By definition of $\psi_F, \quad \Delta_{1,F}$ is the union of all simplices in $\Delta_1,$ containing $F.$ By Lemma \ref{linkflgl}, $\Delta_{1,int}$ is the union by all internal $\mf_l$ of all simplices in $\Delta_1,$ containing $\left\{\mf_l\right\}.$ For $\widetilde{\Delta_1}$ the proof is similar.
\end{proof}

\begin{theorem} \label{widetildeiso}
We have an isomorphism:
$$\phi: \widetilde{\Delta_1} \overset\sim\to \widetilde{\Delta_2},$$
which sends letters in $\MQ$ and in $\MQ'$ to themselves, $\mf_1$ to $\mg_{m_{ij}}$ and $\mf_{m_{ij}}$ to $\mg_1.$
\end{theorem}

\begin{proof}
Consider any subword $\MT$ in $\MQ_1^0.$ If $\MT$ does not contain any letters from $\mw_{i,j}^{0},$ then it is a simplex in $\widetilde{\Delta_1}$ if and only if it gives a simplex in $\widetilde{\Delta_2},$ while being considered as a subword of $\MQ_2^0.$ Indeed, it is a simplex in $\Delta_1$ if its complement $\MT^c$ in $\MQ_1^0$ contains a reduced expression $\mpp$ of $\pi.$ Let $\mq$ be a maximal common subword (in other words, the intersection) of $\mpp$ and $\mw_{i,j}^{0}.$ If $\mq \neq \mw_{i,j}^{0},$ then $\mq$ is a subword of $\mw_{i,j}^{0}$ as well, and $\mpp$ is a subword of the complement of ${{\MT}^c}'$ of $\MT$ in $\MQ_2^0.$ If $\mq = \mw_{i,j}^{0},$ then the word $\mpp',$ obtained from $\mpp$ by the braid move changing $\mw_{i,j}^{0}$ by $\mw_{j,i}^{0},$ is a subword of ${\MT^c}'$ and a reduced expression of $\pi.$ In both cases, $\MT$ is a simplex in $\widetilde{\Delta_2}.$

Assume now that $\MT$ contains $\mf_1,$ but does not contain $\mf_{m_{ij}}.$ The word $\MT'$ obtained from $\MT$ by removing $\mf_1$ is a subword of $\MQ_2^0,$ not containing $\mg_{m_{ij}}.$ Consider a subword $\MT''$ of $\MQ_2^0$ obtained by adding the letter $\mg_{m_{ij}}$ (at its position) to $\MT'.$ It is easy to check that $\MT$ is a simplex in $\widetilde{\Delta_1}$ if and only if $\MT''$ is a simplex in $\widetilde{\Delta_2}.$ Thus, we have
$$\St_{\widetilde{\Delta_1}}(\left\{\mf_1\right\}) \cong \Lk_{\widetilde{\Delta_1}}(\left\{\mf_1\right\}) * \left\{\mf_1\right\}  \cong \Lk_{\widetilde{\Delta_2}}(\left\{\mg_{m_{ij}}\right\}) * \left\{\mg_{m_{ij}}\right\} \cong \St_{\widetilde{\Delta_2}}(\left\{\mg_{m_{ij}}\right\}).$$ 
Similarly, we have $\Lk_{\widetilde{\Delta_1}}(\left\{\mf_{m_{ij}}\right\}) \cong \Lk_{\Delta_2}(\left\{\mg_1\right\})$ and $\St_{\Delta_1}(\left\{\mf_{m_{ij}}\right\}) \cong \St_{\Delta_2}(\left\{\mg_1\right\}).$ 

By Lemma \ref{widetildesense}, any subword in $\MQ_1^0$ providing a simplex in $\widetilde{\Delta_1}$ does not contain internal $\mf_l$'s and contain at most one of $\mf_1$ and $\mf_{m_{ij}}.$ Therefore, the previous arguments show that $\phi(\widetilde{\Delta_1}) \subset \widetilde{\Delta_2}.$ Similarly, $\widetilde{\Delta_2} \subset \phi(\widetilde{\Delta_1}),$ and the statement follows.
\end{proof}

\begin{corollary} \label{delta2iso}
$$\Delta_2 \cong \widetilde{\Delta_1} \sqcup (\Delta_{2,int} \cup \Delta_{2,G}).$$
\end{corollary}

Let $\mathcal{V}$ be the union of the sets of the letters of $\MQ_1^0$ and $\MQ_2^0,$ modulo the identification $\mf_1$ with $\mg_{m_{ij}}$ and $\mf_{m_{ij}}$ with $\mg_1.$

\begin{corollary} \label{subdivident}
We have the identities in complexes considered on the set of vertices $\mathcal{V}:$
$$\Delta_1 \cup (\Delta_{2,int} \cup \Delta_{2,G}) = \Delta_2 \cup (\Delta_{1,int} \cup \Delta_{1,F});$$
\begin{equation} \label{widetildedelta}
(\Delta_1 \backslash \Delta_{1,F}) \cup \Delta_{2,int} = \widetilde{\Delta_1} \cup (\Delta_{1,int} \cup \Delta_{2,int}) = \widetilde{\Delta_2} \cup (\Delta_{1,int} \cup \Delta_{2,int}) = (\Delta_2 \backslash \Delta_{2,G}) \cup \Delta_{1,int}.
\end{equation}
\end{corollary}

\begin{remark}
If $m_{ij} = 2,-$ in other words, if $s_i$ and $s_j$ commute,--  we have $\Delta_{1,int} = \Delta_{2,int} = \emptyset$ and on $\mathcal{V}$ we have $\Delta_{1,F} = \Delta_{2,G}.$ Thus, by Corollary \ref{delta2iso}, we see that $\Delta_1 = \Delta_2$ on $\mathcal{V}.$ This is very natural, and this is the reason why in theory of subword complexes, $\MQ$ is often considered {\it up to commutations}, cf. \cite{PS}.
\end{remark}

\begin{theorem} \label{braidstellar1}
Assume that both conditions ($A_3$) and ($B_3$) are satisfied and $m_{ij} > 3.$
Then either $\Delta_1$ and $\Delta_2$ are isomorphic, or one of them is isomorphic to an edge $(m_{ij} - 2)-$subdivision of the other, or there is a simplicial complex $\widetilde{\Delta}$ which can be obtained from each of $\Delta_i,$ for $i=1,2,$ by an edge $(m_{ij} - 2)-$edge subdivision.

More explicitly,  we have the following cases:
\begin{itemize}
\item[(1)] if both ($A_2$) and ($B_2$) are satisfied, then $\Delta_1$ and $\Delta_2$ are combinatorially isomorphic.
\item[(2)] if ($A_2$) is not satisfied, but ($B_2$) is, then $\Delta_1$ is the $(m_{ij} - 2)-$subdivision of $\Delta_2$ along $G;$ 
\item[(3)] if ($A_2$) is satisfied, but ($B_2$) is not, then $\Delta_2$ is the $(m_{ij} - 2)-$subdivision of $\Delta_1$ along $F;$
\item[(4)] if neither ($A_2$), nor ($B_2$) are satisfied, then there exists a simplicial complex $\widetilde{\Delta}$ which is simultaneously the $(m_{ij} - 2)-$subdivision of $\Delta_1$ along $F$ and the $(m_{ij} - 2)-$subdivision of $\Delta_2$ along $G.$
\end{itemize}
\end{theorem}

Note that if at least one of $\mf_1, \mf_{m_{ij}}$ is not a vertex of $\Delta_1,$ then $F \notin \Delta_1;$ the same holds for $\mg_1, \mg_{m_{ij}}, G$ and $\Delta_2.$

\begin{proof}
We pose $\widetilde{\Delta}$ to be equal to each term in the identities (\ref{widetildedelta}):
$$\widetilde{\Delta} = (\Delta_1 \backslash \Delta_{1,F}) \cup \Delta_{2,int} = (\Delta_2 \backslash \Delta_{2,G}) \cup \Delta_{1,int}.$$
We want to show that if ($B_2$) is satisfied, then $\widetilde{\Delta} = \Delta_1,$ and if ($B_2$) is not satisfied, but ($A_3$) and ($B_3$) hold, then $\widetilde{\Delta}$ is an $(m_{ij} - 2)-$subdivision of $\Delta_1$ along $F.$ With the similar statement concerning $\Delta_2,$ this will prove the theorem.

By Lemma \ref{A2B2}, condition ($B_2$) holds if and only if $\Delta_{1,F} = \emptyset.$ By Lemma \ref{linkflgl} and the definition, this is also equivalent to the identity $\Delta_{2,int} = \emptyset.$ It follows that ($B_2$) holds if and only if $\widetilde{\Delta} = \Delta_1.$ 

Suppose now that ($A_3$) and ($B_3$) hold, but ($B_2$) does not. By Corollary \ref{A3B3edges}, we obtain that $\Lk_{\Delta_2}(\left\{\mg_l, \mg_{l+1}\right\})$ does not contain $\left\{\mg_k\right\},$ for any $k \notin \left\{l, l + 1\right\}.$ Similarly, $\Lk_{\Delta_1}(F)$ does not contain $\left\{\mf_k\right\},$ for any $k \notin \left\{1, m_{ij}\right\}.$ We see that on the vertex set $\mathcal{V},$ we have the following identities (and not just isomorphisms!):
$$\Lk_{\Delta_2}(\left\{\mg_2, \mg_2\right\}) = \Lk_{\Delta_2}(\left\{\mg_2, \mg_3\right\}) =\ldots= \Lk_{\Delta_2}(\{\mg_{(m_{ij} - 1)}, \mg_{m_{ij}}\}) = \Lk_{\Delta_1}(F).$$
It is easy to check now that $(\Delta_1 \backslash \Delta_{1,F})$ is the first term in the definition of $\Sub_{F}^{(m_{ij} - 2)}(\Delta_1),$ while $\Delta_{2,int}$ is the second.
\end{proof}

\begin{theorem} \label{braidstellarmij3}
Assume that $m_{ij} = 3.$ Then the claim of Theorem \ref{braidstellar1} holds, with the same cases (1) - (4). Explicitly, we have the following possibilities:
\begin{itemize}
\item[(1)] if both ($A_2$) and ($B_2$) are satisfied, then $\Delta_1$ and $\Delta_2$ are combinatorially isomorphic.
\item[(2)] if ($A_2$) is not satisfied, but ($B_2$) is, then $\Delta_1$ is the subdivision of $\Delta_2$ along $G;$ 
\item[(3)] if ($A_2$) is satisfied, but ($B_2$) is not, then $\Delta_2$ is the subdivision of $\Delta_1$ along $F;$
\item[(4)] if neither ($A_2$), nor ($B_2$) are satisfied, then there exists a simplicial complex $\widetilde{\Delta}$ which is simultaneously the subdivision of $\Delta_1$ along $F$ and the subdivision of $\Delta_2$ along $G.$
\end{itemize}
\end{theorem}

\begin{proof}
In this case, the vertex set $\mathcal{V}$ consists of all letters in $\MQ\MQ',$ two vertices $\mf_2, \mg_2,$ and two vertices $\mf_1 = \mg_3$ and $\mf_3 = \mg_1.$
With $m_{ij} = 3,$ each of conditions ($A_3$) and ($B_3$) is equivalent to the following condition:
\begin{itemize}
\item[($\ast$)] $\MQ\MQ'$ does not contain any reduced expression of $\pi$.
\end{itemize}
Assume that ($\ast$) holds. We consider the same $\widetilde{\Delta}$ as in the proof of Theorem \ref{braidstellar1}. If ($B_2$) holds, by the same arguments, $\widetilde{\Delta} = \Delta_1.$
Assume now that ($B_2$) does not hold (but ($\ast$) still does). Then $\Lk_{\Delta_1}(\left\{\mf_p, \mf_q\right\})$ does not contain $\left\{\mf_r\right\},$ for $\left\{p, q, r\right\} = \left\{1, 2, 3\right\};$ and the same for $\Delta_2.$ In this case, $\widetilde{\Delta}$ is the subdivision of $\Delta_1$ along $F,$ by the same arguments as in the second case in the proof of Theorem \ref{braidstellar1}.

Suppose now that ($\ast$) is not satisfied. In this case, we have to write down a bit longer formulae. The links of the edges which are interesting to us have the next form:
$$\Lk_{\Delta_1}(\left\{\mf_1, \mf_2\right\}) = \left\{\sigma \in \Delta(\MQ\mf_3\MQ'; \pi), \mf_3 \notin \sigma\right\} \sqcup (\Delta(\MQ\MQ'; \pi) * \left\{\mf_3\right\});$$
$$\Lk_{\Delta_1}(\left\{\mf_2, \mf_3\right\}) =  \left\{\sigma \in \Delta(\MQ\mf_1\MQ'; \pi) = \Delta(\MQ\mf_3\MQ'; \pi), \mf_1 \notin \sigma\right\} \sqcup (\Delta(\MQ\MQ'; \pi) * \left\{\mf_1\right\});$$
$$\Lk_{\Delta_1}(F) =  \left\{\sigma \in \Delta(\MQ\mf_2\MQ'; \pi), \mf_2 \notin \sigma\right\} \sqcup (\Delta(\MQ\MQ'; \pi) * \left\{\mf_2\right\}).$$
Note that, on the vertex set $\mathcal{V},$ we have an identity
$$\left\{\sigma \in \Delta(\MQ\mf_3\MQ'; \pi), \mf_3 \notin \sigma\right\} = \left\{\sigma \in \Delta(\MQ\mf_1\MQ'; \pi), \mf_1 \notin \sigma\right\} = \left\{\sigma \in \Delta(\MQ\mg_2\MQ'; \pi), \mg_2 \notin \sigma\right\} =: \Delta',$$
where $\Delta'$ is the full subcomplex of $\Delta(\MQ\ms_j\MQ'; \pi)$ on the vertex $\mathcal{V}_{Q,Q'}$ consisting of all letters in $\MQ\MQ'.$
Similarly, we pose
$$\Delta'' := \left\{\sigma \in \Delta(\MQ\mf_2\MQ'; \pi), \mf_2 \notin \sigma\right\} = \left\{\sigma \in \Delta(\MQ\mg_1\MQ'; \pi), \mg_1 \notin \sigma\right\} = \left\{\sigma \in \Delta(\MQ\mg_3\MQ'; \pi), \mg_3 \notin \sigma\right\}$$
to be the full subcomplex of $\Delta(\MQ\ms_j\MQ'; \pi)$ on the vertex $\mathcal{V}_{Q,Q'}.$
We conclude that
$$\Delta_1 = \widetilde{\Delta_1} \sqcup (\left\{\mu \cup \left\{\mf_1,\mf_2,\mf_3\right\} | \mu \in \Delta(\MQ\MQ'; \pi)\right\}) \sqcup \left\{\sigma \cup F | \sigma \in \Delta''\right\} \sqcup$$
$$\sqcup \left\{\rho \cup \left\{\mf_1, \mf_2\right\}, \rho \cup \left\{\mf_2\right\}, \rho \cup \left\{\mf_2, \mf_3\right\} | \rho \in \Delta' \right\}.$$
Then its subdivision along $F,$ where we take $\mg_2$ as the new vertex, has the following form:

$$\Sub_{F}(\Delta_1) = \widetilde{\Delta_1} \sqcup \left\{\mu \cup \left\{\mf_1,\mf_2,\mg_2\right\},  \mu \cup \left\{\mf_3,\mf_2,\mg_2\right\}| \mu \in \Delta(\MQ\MQ'; \pi)\right\}  \sqcup $$
\begin{equation} \label{subfdelta1mij3}
\sqcup \left\{\sigma \cup \left\{\mf_1, \mg_2\right\}, \sigma \cup \left\{\mg_2\right\}, \sigma \cup \left\{\mg_2, \mf_3\right\} | \sigma \in \Delta''\right\} \sqcup
\end{equation}
$$\sqcup \left\{\rho \cup \left\{\mf_1, \mf_2\right\}, \rho \cup \left\{\mf_2\right\}, \rho \cup \left\{\mf_2, \mf_3\right\} | \rho \in \Delta' \right\}.$$
By similar arguments, we have
$$\Sub_{G}(\Delta_2) = \widetilde{\Delta_2} \sqcup \left\{\mu \cup \left\{\mg_1,\mg_2,\mf_2\right\},  \mu \cup \left\{\mg_3,\mg_2,\mf_2\right\}| \mu \in \Delta(\MQ\MQ'; \pi)\right\}  \sqcup$$
\begin{equation} \label{subgdelta2mij3}
\sqcup \left\{\sigma \cup \left\{\mg_1, \mg_2\right\}, \sigma \cup \left\{\mg_2\right\}, \sigma \cup \left\{\mg_2, \mg_3\right\} | \sigma \in \Delta''\right\} \sqcup
\end{equation}
$$ \sqcup \left\{\rho \cup \left\{\mg_1, \mf_2\right\}, \rho \cup \left\{\mf_2\right\}, \rho \cup \left\{\mf_2, \mg_3\right\} | \rho \in \Delta' \right\}.$$
Recall that we identify in $\mathcal{V}$ the vertices: $\mf_1$ with $\mg_3$ and $\mf_3$ with $\mg_1;$ moreover, on this set, by Theorem \ref{widetildeiso}, we have $\widetilde{\Delta_1} = \widetilde{\Delta_2}.$ Now it is easy to observe that the right hand sides in (\ref{subfdelta1mij3}) and in (\ref{subgdelta2mij3}) are the same; thus, $\Sub_F(\Delta_1) = \Sub_G(\Delta_2).$
\end{proof}

\begin{example}
Let us consider the group $I_2(m),$ for $m \geq 3.$ It is generated by two elements $s_1, s_2,$ modulo the relations $(s_1 s_2)^m = (s_2 s_1)^m = \me.$ Let us consider a pair of words 
$$Q_1 = \ms_1\ms_2\ms_1\ms_2\ms_1\ms_2\ldots, \quad Q_2 = \ms_1\ms_2\ms_2\ms_1\ms_2\ms_1\ldots$$
of length $(m + 2).$ They have the form $Q_1^0$ and $Q_2^0$ for $Q = \ms_1\ms_2, Q' = \me, i = 1, j = 2.$ Let us take $\pi = w_o \in I_2(m).$ Since the words 
$\ms_1\ms_2\ms_1\ms_2\ldots, \quad \mbox{and} \quad \ms_2\ms_1\ms_2\ms_1\ldots$
are the only reduced expressions of $w_o,$ it is easy to check that condition ($A_2$) is not satisfied, but ($B_2$) holds. It follows that $\Delta_1 = \Delta(Q_1, w_o)$ is the edge $(m - 2)-$subdivision of $\Delta_2 = \Delta(Q_2, w_o).$ Explicitly, $\Delta_2$ is the boundary complex of the rectangle: indeed, precisely the first three and the last one of the letters in $Q_2$ provide vertices of $\Delta_2.$ Then $\Delta_1$ is the boundary complex of the $(m + 2)-$gon, and indeed, all the letters in $Q_1$ correspond to vertices of $\Delta_2.$ In fact, $\Delta_1$ is just the $(\ms_1\ms_2)-$cluster complex of type $I_2(m)$ from Example \ref{cluster}.
In this way, we realize all the boundary complexes of $n-$gons, for $n \geq 4,$ as subword complexes. 
\end{example}

\begin{theorem} \label{braidhgamma}
Assume that either conditions ($A_3$) and ($B_3$) are satisfied, or $m_{ij} = 3,$ or all of this holds. Then we have
\begin{itemize}
\item[(i)]
$H(\Delta_2) - H(\Delta_1) = (m_{ij} - 2)\alpha t\left(H(\Delta(\MQ_2^2; \pi)) - H(\Delta(\MQ_1^{2}; \pi))\right);$
\item[(ii)]
If $\Delta_1, \Delta_2$ are spherical, we have
$$\gamma(\Delta_2) - \gamma(\Delta_1) = (m_{ij} - 2)\tau\left(\gamma(\Delta(\MQ_2^{2}; \pi)) - \gamma(\Delta(\MQ_1^{2}; \pi))\right).$$
\end{itemize}
\end{theorem}

\begin{proof}
By Fact \ref{polysubdiv}, the statement follows from Theorem \ref{braidstellar1} and the existence of the isomorphisms $\varphi_G$ and $\psi_F.$
\end{proof}

\begin{corollary} \label{braidgal}
Assume that ($A_2$) is satisfied. Then we have
\begin{itemize}
\item[(i)] if $\Delta_1$ and $\Delta(\MQ_2^2; \pi)$ are flag, then so is $\Delta_2.$
\item[(ii)] if Conjecture \ref{galconj} holds for $\Delta_1$ and for $\Delta(\MQ_2^2; \pi),$ then it also holds for $\Delta_2.$
\end{itemize}
\end{corollary}

\begin{corollary} \label{sequence}
Consider any two expressions $\MQ_1, \MQ_2$ of an element $Q \in W,$ related by a sequence of braid moves, each of those satisfies conditions of Theorem \ref{braidstellar1} or of Theorem \ref{braidstellarmij3}. For any element $\pi \in W,$ the subword complex $\Delta(\MQ_2; \pi)$ can be obtained from $\Delta(\MQ_1; \pi)$ by a sequence of edge subdivisions and inverse edge subdivisions.
\end{corollary}

We can reformulate the above results in terms of $2-$truncations of polytopes.

\begin{theorem} \label{braid2trunc}
Assume that either conditions ($A_3$) and ($B_3$) are satisfied, or $m_{ij} = 3,$ or all of this holds. Assume that $\Delta_1$ and $\Delta_2$ can be realized as the nerve complexes of polytopes $P_1$ and $P_2,$ respectively. Then we are in one of the following cases:
\begin{itemize}
\item[(1)] if both ($A_2$) and ($B_2$) are satisfied, then $P_1$ and $P_2$ are combinatorially isomorphic.
\item[(2)] if ($A_2$) is not satisfied, but ($B_2$) is, then $P_1$ can be obtained from $P_2$ by a sequence of $(m_{ij} - 2) \quad 2-$truncations at $G;$
\item[(3)] if ($A_2$) is satisfied, but ($B_2$) is not, then $P_2$ can be obtained from $P_1$ by a sequence of $(m_{ij} - 2) \quad 2-$truncations at $F;$
\item[(4)] if neither ($A_2$), nor ($B_2$) are satisfied, then there exists a simple polytope $\widetilde{P}$ which can be obtained simultaneously (a) from $P_1$ by a sequence of $(m_{ij} - 2) \quad 2-$truncations at $F;$ and (b) from $P_2$ by a sequence of $(m_{ij} - 2) \quad 2-$truncations at $G.$
\end{itemize}
\end{theorem}

\begin{corollary} \label{A2poly}
Assume that ($A_2$) is satisfied. Then we have:
\begin{itemize}
\item[(i)] if $\Delta_1$ can be realized as the nerve complex of a simple polytope $P_1,$ then $\Delta_2$ can be realized as the nerve complex of a simple polytope $P_2;$
\item[(ii)] if $P_1$ is flag, then so is $P_2.$
\end{itemize}
\end{corollary}

\begin{corollary}
Assume that the group $W$ has a simply-laced Coxeter diagram, i.e. $W$ is a finite free product of groups of types $A_n, n \geq 1; D_n, n \geq 4, E_6, E_7$ or $E_8.$ Then the statements of Theorems \ref{braidstellarmij3}, \ref{braidhgamma}, \ref{braid2trunc}, and Corollaries \ref{braidgal}, \ref{sequence} and \ref{A2poly} hold for any words and any braid move.
\end{corollary}

\begin{example}
Let us consider the group $A_3.$ We will describe a sequence of braid moves and corresponding transformations of the polytopes dual to corresponding subword complexes. Here the group element $\pi$ whose reduced expressions we count is always taken to be $w_o.$ We have:
$$\xymatrix@R=0.3cm{*+++{\underline{\ms_1}\underline{\ms_2}\underline{\ms_3}\underline{\ms_3}\ms_2\ms_1\underline{\ms_3}\ms_2\underline{\ms_3}} \ar@{->}[d] && *+++{I^3} \\
*+++{\underline{\ms_1\ms_2\ms_3\ms_3}\ms_2\underline{\ms_3}\ms_1\ms_2\underline{\ms_3}} \ar@{->}[d] && *+++{I^3} \\
*+++{\underline{\ms_1\ms_2\ms_3\ms_2\ms_3\ms_2}\ms_1\ms_2\underline{\ms_3}} \ar@{->}[d] && *+++{I\times As^2} \\
*+++{\underline{\ms_1\ms_2\ms_3\ms_2\ms_3}\ms_1\ms_2\underline{\ms_1\ms_3}} \ar@{->}[d] && *+++{I\times As^2} \\
*+++{\underline{\ms_1\ms_2\ms_3\ms_2}\ms_1\underline{\ms_3}\ms_2\underline{\ms_3\ms_1}} \ar@{->}[d] && *+++{I\times As^2} \\
*+++{\underline{\ms_1\ms_2\ms_3\ms_2}\ms_1\underline{\ms_2\ms_3\ms_2\ms_1}} \ar@{->}[d] && *+++{P^3} \\
*+++{\underline{\ms_1\ms_2\ms_3\ms_1\ms_2\ms_1\ms_3\ms_2\ms_1}} \ar@{->}[d] && *+++{As^3} \\
*+++{\underline{\ms_1\ms_2\ms_3\ms_1\ms_2\ms_3\ms_1\ms_2\ms_1}}  && *+++{As^3.} \\
}$$
In each row, on the right, we give the polytope dual to the complex $\Delta(\MQ, w_o),$ where $\MQ$ is the word on the left. The underlined letters are those who provide vertices of $\Delta(\MQ, w_o)$ and facets of this polytope. Here $As^n$ is the $n-$dimensional associahedron; in particular, $As^2$ is the pentagon. $P^3$ is the polytope obtained from the product $I\times As^2$ by the $2-$truncation at an edge belonging to one of the bases. The last word and the last $As^3$ come from Example \ref{cluster}.
\end{example}

Let us now consider the words of the form $\MQ \mpp \MQ',$ where $\MQ, \MQ'$ are arbitrary words in $S^*, \mpp$ is a reduced expression of a fixed element $\pi$ of the group $W.$ By Fact \ref{wordproperty}, any two reduced expressions $\mpp, \mpp'$ of $\pi$ are related by a sequence of braid moves; therefore, any two words $\MQ \mpp \MQ', \MQ \mpp' \MQ'$ of this form are related by a sequence of braid moves. Thus, by Corollary \ref{sequence}, under certain conditions, the subword complexes $\Delta(\MQ \mpp \MQ'; \pi)$ and $\Delta(\MQ \mpp' \MQ'; \pi)$ are related to each other by a sequence of edge subdivisions and inverse edge subdivisions. Let $\mathcal{P}$ be a set of all reduced expressions of $\pi.$

\begin{definition}
For any words $\MQ, \MQ' \in S^*,$ we define a partial order $\rho_{\MQ,\MQ',\pi}$ on $\mathcal{P}$ by the following rule: for $\mpp, \mpp'\in \mathcal{P},$ we have $\rho_{\MQ,\MQ',\pi}(\mpp, \mpp')$ if $\Delta(\MQ \mpp \MQ'; \pi)$ can be obtained (up to combinatorial isomorphism) from $\Delta(\MQ \mpp' \MQ'; \pi)$ by a sequence of edge subdivisions.
\end{definition}

\begin{problem}
Can we characterize all the triples $(\MQ, \MQ', \pi),$ such that the order $\rho_{\MQ,\MQ',\pi}$ is a join- or meet-semilattice? 
\end{problem}

In the upcoming sequel of this work \cite{Go2}, we shall discuss properties of the order $\rho_{\mc,\me,w_o},$ for $\mc$ a reduced expression of a Coxeter element in $W.$

\end{document}